\documentclass[12pt]{amsart}
\usepackage{amsmath,color}
\usepackage{amsmath}
\usepackage{amssymb}
\usepackage{amscd}
\usepackage[citecolor = green]{hyperref}
\usepackage{amsrefs}
\usepackage{longtable}
\usepackage{graphicx}
\usepackage{comment}
\usepackage{tikz-cd}

\newtheorem{thm}{Theorem}[section]
\newtheorem*{thm*}{Theorem}
\newtheorem{prop}{Proposition}[section]

\newtheorem{lem}{Lemma}[section]

\newtheorem{conj}{Conjecture}[section]

\newtheorem{example}{Example}[section]

\theoremstyle{remark}

\theoremstyle{plain} \newtheorem{GS2.1}{Theorem A}

\renewcommand{\geq}{\geqslant}
\renewcommand{\leq}{\leqslant}
\renewcommand{\mod}[1]{{\ifmmode\text{\rm\ (mod~$#1$)}\else\discretionary{}{}{\hbox{ }}\rm(mod~$#1$)\fi}}

\newcommand{\codim}{\mathrm{codim}}
\newcommand{\id}{\mathrm{id}}

\begin{document}

\title{Strict Bott-Samelson Resolutions of Schubert Varieties}

\author{Sergio Da Silva}

\address{Cornell University, Ithaca NY}

\email{smd322@cornell.edu}

\thanks{Research supported in part by NSERC PGS-D3 Scholarship}

\begin{abstract}

We explore the relationship between two desingularization techniques for Schubert varieties. The Bott-Samelson resolution is the more common of the two, but it fails to encompass many properties that Hironaka resolutions provide, in particular, being an isomorphism over the smooth locus. Using a computer search, a list of cases where Bott-Samelson resolutions having this ``strictness" property is compiled for the $n=5,6$ cases. A conjecture based on these results is formulated and is subsequently verified for $n=7, 8$. A comparison between Bott-Samelson resolutions and blow-ups is also provided.

\end{abstract}

\maketitle

\tableofcontents

In 1958, Bott and Samelson introduced certain spaces \cite{Bott-Samelson} that provided convenient desingularizations of Schubert varieties (which were later generalized by Hansen \cite{Hansen} and  Demazure \cite{Demazure}). Around the same time, Hironaka published his famous result on the resolution of singularities for algebraic varieties in characteristic zero \cite{Hironaka}. His result could of course be applied to Schubert varieties to obtain a second, very different sort of desingularization. While Hironaka's method is a stronger result in general, the fact that Bott-Samelson resolutions are the preferred desingularization method would speak to the combinatorial conveniences they possess. What if we could get the best of both methods?

 An important feature of a Hironaka desingularization (such as the algorithm in \cite{Bierstone-Milman}) is the fact that the desingularization map is an isomorphism over the smooth locus. One calls such a desingularization a \textbf{strict resolution} of singularities. In general, the Bott-Samelson resolution is not an isomorphism over the smooth locus of a Schubert variety. Even more, while Hironaka's method utilizes blow-ups, a Bott-Samelson resolution is not in general a blow-up map (see Section \ref{blow-up}). Nevertheless, the resolution has many combinatorial properties that are natural for working with Schubert varieties. There is for example an action of a torus $T$ with isolated fixed points, and the map on $T$-fixed points is especially easy to utilize. An explanation of some of these useful properties can be found in Section \ref{BS-Resolution}.
 
 Similar constructions to Bott-Samelson varieties can be considered for the purpose of strict resolutions. For example, in \cite{Cortez}, Cortez introduces quasi-resolutions of Schubert varieties and studies their singularities in terms of intersections of the exceptional loci. This approach extends the work in \cite{Cortez2} where strict resolutions are provided for covexillary Schubert varieties. For the general case however, these quasi-resolutions provide a partial desingularization that may not necessarily be strict. Some of these details are discussed in Section \ref{BSR} 

We will explore both desingularization algorithms in more detail, and suggest a more general method for resolving singularities which is close to the Bott-Samelson construction, and yet is a strict desingularization. We can do this by taking advantage of a generalized Bott-Samelson resolution (see Section \ref{BSR}). In Section \ref{strict} we implement a computer program to verify which Schubert varieties can be resolved strictly using this generalized method. We will verify the following result:

\begin{thm*}
	Let $X^w$ be a singular Schubert variety. If $w\in S_5$, then there exists a strict Bott-Samelson resolution of singularities for $X^w$ iff  $w\neq 45312$. If $w\in S_6$, then there exists a strict Bott-Samelson resolution of singularities for $X^w$ if  $w\ngeq 156423, 453126 \text{ or } 632541$.
\end{thm*}

Unfortunately the converse of the theorem for the $n=6$ case does not hold (see Example 4.2).
The longer list of exceptions in the $n=6$ case mostly relate to a pattern embedding of $45312$, which suggested a recursive construction for such a list in general. The conjecture for $n>4$ in Section 4 was formulated based on these results. Since there are other reasonable interpretations of the $n=5,6$ computer results, it seemed necessary to verify that the conjecture was true at least for $n=7, 8$ and indeed this is the case.

 The curiosity of whether a Bott-Samelson resolution could be viewed as a blow-up (bringing the technique closer to Hironaka's version) led to the work in Section 3. While no general theorems are provided in this section, a new viewpoint of the resolution is explored and some groundwork is provided for further research.

\textbf{Acknowledgments:} I would like to thank my advisor Allen Knutson for our many discussions on this topic. This article would not be possible without the ideas that he shared with me.

\section{Schubert and Bott-Samelson Varieties}\label{BS-Resolution}

Let $G=GL_n(\mathbb{C})$ be the group of invertible $n\times n$ matrices with entries in $\mathbb{C}$. Let us also denote the subgroup of upper triangular, lower triangular and diagonal matrices by $B, B_{-}, T$ respectively. The space of complete flags Flag($\mathbb{C}^n$) can be identified as the quotient $G/B$. We know that $G/B$ has a Bruhat decomposition

$$G/B = \bigsqcup_{w\in S_n} BwB/B \text{  .}$$

For each $w\in W=S_n$, let $X^w_o = BwB/B$, called a Schubert cell. Then its Zariski closure is the Schubert variety $X^w$ (this is different than the notation $X_w$ used by some other authors). There is an obvious left $T$ action on $X^w$, and the $T$-fixed points are $e_v :=vB/B$ for $v\leq w$. Every point in $X^w$ is contained in the $B$-orbit of some $e_v$.

 We should note that except for sections that utilize the explicit pattern embedding characterization of singularities in $GL_n(\mathbb{C})$ (such as Sections \ref{n5} and \ref{n6}), the results in this paper are valid for any reductive group.

\subsection{Kazhdan-Lusztig Varieties}

To get local coordinates for $G/B$, and hence to view local equations for $X^w$, we resort to computing Kazhdan-Lusztig ideals (see \cite[\S 3]{Woo-Yong}). The $B$-action on $G/B$ provides us with an isomorphism between a neighborhood of any point with a neighborhood of a $T$-fixed point. Let $X_w^o:= B_{-}wB/B$ be the opposite Schubert cell. Then an affine neighborhood of $e_v$ is given by $vX^o_{\id}$. To provide local equations for $X^w$ at $e_v$, it is enough to study $X^w\cap vX^o_{\id}$. By the Kazhdan-Lusztig Lemma (see 3.2 in \cite{Woo-Yong}), this intersection is isomorphic to $(X^w\cap X^o_v)\times \mathbb{A}^{l(v)}$. The $ X^w\cap X^o_v$ provides us with the local equations we are looking for. These equations will be useful for computations in the examples that follow.

 As a side note, the fact that $ X^w\cap X^o_v\neq \emptyset$ iff $w\geq v$ (and conventionally posets are drawn with larger elements on the top) helps justify the notation $X^w$ instead of $X_w$.
 
\begin{example}
	\normalfont Let us compute local equations for the variety $X^{4231}$ at $e_{\id}$. To do this we need to intersect the Schubert conditions from the permutation matrix

 \[ \left( \begin{array}{cccc}
 0 & 0 & 0 & 1\\
 0 & 1 & 0 & 0 \\
 0 & 0 & 1 & 0 \\
 1 & 0 & 0 & 0 \end{array} \right)\]
(which provide defining equations for $\pi^{-1}(X^{4231})$ where $\pi:G\rightarrow G/B$) with the coordinates from the open cell (which is isomorphic to affine space)

 \[ \left( \begin{array}{cccc}
 1 & 0 & 0 & 0\\
 z_{31} & 1 & 0 & 0 \\
 z_{21} & z_{22} & 1 & 0 \\
 z_{11} & z_{12} & z_{13} & 1 \end{array} \right)\]

 The Schubert conditions simply say that the southwest-most $2\times 2$ minor vanishes. In other words, the Kazhdan-Lusztig ideal is generated by $z_{21}z_{12}-z_{22}z_{11}=0$. 
 \end{example}

\subsection{Bott-Samelson Resolutions}\label{BSR}

Let $Q=(w_1,...,w_k)$, where $w_i\in W$. A generalized Bott-Samelson variety $BS^Q$ is the quotient of $\overline{Bw_1B}\times ...\times \overline{Bw_kB}$ by the $B^k$ action given by:

$$(b_1,...,b_k)\cdot (p_1,...,p_k):= (p_1b_1^{-1},b_1p_2b_2^{-1},...,b_{k-1}p_kb_k^{-1}).$$

We write this as $BS^Q:= \overline{Bw_1B}\times^B ...\times^B \overline{Bw_kB}/B$ where $\times^B$ is the quotient of the pair by the above action for $k=2$. We naturally have the $B$-equivariant map

$$\varphi_Q:\overline{Bw_1B}\times^B ...\times^B \overline{Bw_kB}/B \rightarrow G/B$$
defined by $(p_1,...,p_k) \rightarrow p_1\cdot ... \cdot p_k B/B$. It is not hard to see that $\varphi_Q$ maps onto $X^{\text{Dem(Q)}}$, where Dem(Q) is the Demazure product of the entries of $Q$. We can define Dem(Q) inductively as

\[
\text{Dem((Q,$s_i$))} = \begin{cases} 
\text{Dem(Q)$\cdot s_i$} & l(\text{Dem(Q$\cdot s_i$)}) > l(\text{Dem(Q)}) \\
\text{Dem(Q)} & \text{otherwise} \\
\end{cases}\]
where the Demazure product of the empty word is the identity.

We can similarly define a sub Bott-Samelson variety $BS^R\subset BS^{Q}$ by taking $R=(v_1,...,v_k)$ where $v_i\leq w_i$.

In many papers on the subject, the Bott-Samelson variety refers to the special case where each non-identity $w_i$ is a simple reflection $s_{k_i}\in W$, so that $\overline{Bw_iB}\cong P_{k_i}$ is a minimal parabolic subgroup of  $G$. In this case, $BS^Q$ is smooth (indeed it is just an iterated $\mathbb{P}^1$-bundle) and so $\varphi_Q$ defines a resolution of singularities for $X^w$.  This version of $BS^Q$ is easier to visualize in the context of flag varieties (see \cite{Escobar} for examples of this). Note that in this case, we will write $Q$ as a string of $k_i$ with $-$ in place of the identity (for example, $Q= (s_1,s_2,e,s_1)$ can be written as $12-1$).

Although $BS^Q$ is not smooth in general, it is Cohen-Macaulay and normal.

\begin{lem}
The Bott-Samelson variety $BS^Q$ is both Cohen-Macaulay and normal.
\end{lem}

\begin{proof}

Using the projection map

$$\overline{Bw_1B}\times^B ...\times^B \overline{Bw_kB}/B \rightarrow \overline{Bw_1B}\times^B ...\times^B \overline{Bw_{k-1}B}/B $$
it is easy to see that $BS^Q$ is just an iterated Schubert variety bundle over $X^w$. It is well known that $X^w$ is both Cohen-Macaulay and normal. Since both of these conditions are local, and since the bundle map above is locally trivial, it follows by induction that $BS^Q$ also has these properties.
\end{proof}

\begin{lem}
 Let $\varphi_Q$ be a Bott-Samelson map with $Q=(w_1,...,w_k)$. Then $\varphi_Q$ is always a proper map. Furthermore, $\varphi_Q$ is birational if and only if $Dem(Q) = \prod w_i$ (in which case we say $Q$ is reduced).
\end{lem}

\begin{proof}

 Since $BS^Q$ and $X^w$ are projective, they are both proper over $\mathbb{C}$ so $\varphi_Q$ must be proper also. The second statement follows from the fact that $\varphi_Q$ is an isomorphism away from the sub Bott-Samelson varieties $BS^R$ where $R$ is not reduced. Then $Q$ is reduced iff $R\neq Q$.

\end{proof}

Let us fix notation and denote the proper subvariety where $\varphi_Q$ is not an isomorphism by $BS^Q_{ni}$. In fact, $$BS^Q_{ni} = \bigcup_{R \text{ not reduced}}BS^R.$$

The caveat to using $\varphi_Q$ is that while we are afforded an easy and explicit way to construct a birational map (for reduced $Q$) from a smooth variety to a Schubert variety, this map is not at all like a Hironaka resolution of singularities. The map does enjoy useful properties like being $B$-equivariant, but it is not in general an isomorphism away from the singular locus.

A similar object was considered in  \cite{Cortez} to provide a geometric description for Schubert variety singularities. Given two parabolic subgroups $B\subset P_J\subset P_I$ determined by certain conditions involving $X^w$, we can define

$$ P_I\times^{P_J} X^v \rightarrow X^w $$
where $X^v\subset X^w$. The map onto $X^w$ is not in general an isomorphism over the smooth locus, but the singular locus of $X^w$ can be described by the intersection of certain exceptional loci. Even though $P_I\times^{P_J} X^v$ may not be smooth, its singularities are simpler and constitutes an improvement.

\subsection{Schubert Variety Singularities}

There is a well-known way to tell whether a Schubert variety is singular: $X^w$ is smooth iff it avoids the patterns $3412$ and $4231$. While this answers the question about which Schubert varieties are singular, we need to know where they are singular if we are going to talk about $\varphi_Q$ being strict.

This question can be answered using interval pattern avoidance \cite[\S 2]{Woo-Yong}. Theorem 6.1 in the same article provides the necessary intervals for finding the singular locus. Roughly put, given an interval $[u,v]$ on the list from the theorem ($u\leq v$ are both permutations), if $v$ occurs as a pattern in $w$, then a similarly embedded $u$ provides a piece of the singular locus of $X^w$. 

As an example, $[2143,4231]$ is on the list mentioned above. Since $4231$ is clearly a subpattern of itself, the singular locus of $X^{4231}$ contains the Schubert subvariety $X^{2143}$. The only other applicable interval on that list is $[1324,3412]$, and since $3412$ is not a subpattern of $4231$, then $X^{2143}$ is in fact the entire singular locus of $X^{4231}$.

The details about interval pattern avoidance are not important for the purpose of the sections that follow. These details were necessary for implementing computer code which checked which $X^w$ could be resolved using strict $\varphi_Q$.

\section{Hironaka's Desingularization}\label{Resolution Algorithm}

Given a variety $X$ defined over an algebraically closed field of characteristic $0$, there exists a sequence of blow-ups

 \begin{align*}
  X=X_0\stackrel{\sigma_1}{\leftarrow} X_1\stackrel{\sigma_2}{\leftarrow}\ldots\stackrel{\sigma_k}{\leftarrow} X_k\stackrel{\sigma_{k+1}}{\leftarrow}\ldots\stackrel{\sigma_t}{\leftarrow} X_t
 \end{align*}
 such that $X_t$ is smooth (here $X_{j+1}$ denotes the strict transform of $X_j$) and each $\sigma_{j+1}$ has smooth center $C_j\subset$ Sing$(X_j)$ (see \cite{Bierstone-Milman}). 

   The statement of a Hironaka desingularization usually includes conditions on the exceptional divisors (such as being simple normal crossings with the strict transform).  We will ignore these for now  and instead focus on two main features. First, the center $C_j$ is contained in the singular locus of $X_j$. Second, the desingularization map is a sequence of blow-up maps. The reader will probably notice that the Bott-Samelson resolution has neither of these properties, so a comparison between each desingularization method might be difficult to arrange.  Instead we will focus on finding a suitable subvariety of $X^w$ to call the center of $\varphi_Q$. \newline

 \textbf{Universal Property of Blow-ups:} The blow-up of a scheme $X$ along some closed subscheme $Y$ is a fibre diagram:

  \[
  \begin{tikzcd}
  E \arrow[swap]{d}\arrow{r} & Bl_Y(X)\arrow{d}\\
  Y \arrow{r} & X
  \end{tikzcd}
  \]
 where $E$ is an effective Cartier divisor, such that any other map to $X$ with an effective Cartier divisor as the fibre over $Y$ factors through the diagram. The only Cartier divisors considered in the sections that follow are effective.

 Characterizing the blow-up in this way will be useful in determining whether we can view Bott-Samelson maps as blow-up maps.

Let us consider a motivational example to observe the similarities between each desingularization technique. Consider $X^{4231}$, one of the two simplest singular Schubert varieties. The permutation $4231$ can be written as the word $s_1s_2s_3s_2s_1$ in simple reflections. Let $Q =(s_1,s_2s_3s_2s_1)$.  It is not hard to show that

$$\varphi_Q:\overline{Bs_1B}\times^B \overline{Bs_2s_3s_2s_1B}/B \rightarrow X^{4231}$$
is a strict Bott-Samelson resolution. As we will see later in Section \ref{strict}, it is enough to check that this is true at $T$-fixed points. In other words, we are examining the product map on Bruhat intervals

$$[e,s_1]\times [e,s_2s_3s_2s_1]\rightarrow [e,s_1s_2s_3s_2s_1].$$

We observe that this map is 2-to-1 over $\{e,s_1,s_3,s_1s_3\}$ (for example, $(e,s_3)$ and $(s_1,s_3s_1)$ are the fibre over $s_3$) and hence one-to-one everywhere else by a cardinality check. By \cite{Woo-Yong}, $X^{4231}$ is singular precisely at the $T$-fixed points $\{e,s_1,s_3,s_1s_3\}$. Therefore $\varphi_Q$ is a strict Bott-Samelson resolution of singularities (notice that $BS^Q$ is smooth).

To observe what a Hironaka resolution would do (at least locally), we should work in an affine neighborhood on the Schubert variety, and compute some Kazhdan-Lusztig varieties. In the previous section, we showed that the Kazhdan-Lusztig ideal of $X^{4231}$ near the identity is generated by $z_{21}z_{12}-z_{22}z_{11}=0$. The center of blowing-up in a Hironaka resolution must then be the origin. It is also easy to see that the strict transform is smooth.

How does this relate back to the Bott-Samelson construction? Observe that the $T$-fixed points for which the map was not injective were $\{1,s_1,s_3,s_1s_3\}$. This means that $\varphi_Q$ is injective away from $X^w$ for $w\leq 2143$. We can compute the Kazhdan-Lusztig ideal for the singular locus $X^{2143}$ (in the same chart $X^o_{\id}$) to observe that it is locally defined by $z_{11}=z_{12}=z_{21}=z_{22}=0$, the origin. This corresponds to the same center as a Hironaka desingularization algorithm!

Although this local picture is encouraging, we have to remember that $\varphi_Q$ is not the same thing as the blow-up map at the origin. In fact, the preimage of the singular locus in $\varphi_Q$ is not even a divisor in $BS^Q$!

This suggests that if we want to make the Bott-Samelson resolution look more like a blow-up, we will have to also blow-up points where the map is already one-to-one in order to get a higher dimensional fibre. If we can find a subvariety $Y$ in $X^w$ which contains the singular locus such that the fibre of $\varphi_Q$ over $Y$ is a divisor, then the Bott-Samelson resolution will be one step closer to behaving like a blow-up.

 \section{Comparing Bott-Samelson Resolutions to Blow-ups}\label{blow-up}

To make a fair comparison between the desingularization maps we are considering, we begin by asking what relationship a blow-up map has with the Bott-Samelson resolution. A simple application of the universal property for blow-ups gives us a starting point.

\begin{lem}
	Let $\varphi:X\rightarrow Z$ be a proper birational map. Then there exists a Cartier divisor $D\subset X$ such that $\varphi$ factors through a blow-up of $Z$ along $Y=\varphi(D)$. 
\end{lem}
\begin{proof}
	Since $\varphi$ is birational, it is generically one-to-one. Therefore there is a subvariety $X_{ni}\subset X$ such that $\varphi$ is an isomorphism everywhere except on $X_{ni}$.
	We can find a Cartier divisor $D$ containing $X_{ni}$ (the product of the local equations for $X_{ni}$ provide a local defining equation for $D$). Since $\varphi$ is proper, $Y$ is a closed subvariety of $Z$. Now $\varphi$ is an isomorphism on $D\setminus X_{ni}$, and $X_{ni}$ is the full preimage of its image; therefore $\varphi^{-1}(Y)=D$. By the universal property of blow-ups, $\varphi$ factors through the blow-up of $Z$ along $Y$, as required. 
\end{proof}
\begin{prop}
 Let $\varphi_Q:BS^Q\rightarrow X^w$ be a Bott-Samelson resolution of $X^w$ where $Q$ is a reduced word for $w$. Then there exists a Cartier divisor $D\subset BS^{Q}$ such that $\varphi_Q$ factors through a blow-up of $X^w$ along $\varphi_Q(D)$. Furthermore, $\codim(\varphi_Q(D))\geq 2$ iff $D=BS^Q_{ni}$.
\end{prop}
\begin{proof}
The existence of $D$ follows from Lemma 1.2 and Lemma 3.1.

Suppose $D = BS^Q_{ni}$. Notice that $\codim(\varphi_Q(BS^R))\geq 2$ if $R$ is not reduced. Since $BS^Q_{ni}$ is the union of such $BS^R$, $\codim(\varphi(D))\geq 2$.

Next suppose that $\codim(\varphi_Q(D))\geq 2$. If $D\setminus BS^Q_{ni}\neq \emptyset$, then $\varphi_Q(D)$ has codimension one in $Y$ since $\varphi_Q$ is an isomorphism away from $BS^Q_{ni}$ (and $BS^Q$ and $X^w$ have the same dimension). This contradicts our assumption, hence $D = BS^Q_{ni}$.
\end{proof}

Unfortunately, we cannot claim that $\varphi_Q$ is itself the blow-up of $X^w$ along $Y$ since $\varphi_Q$ may not satisfy the universal property. One barrier to this is when $D$ from Proposition 3.1 is not irreducible. In this case, it might be possible to blow-down one of the irreducible components of $D$ to get another map where the inverse image of $Y$ is a Cartier divisor. To put the mind of the reader at ease as to whether such a situation occurs, the next proposition provides such an example. Therefore, when $D$ necessarily has more than one irreducible component, it is not obvious under what conditions $\varphi_Q$ is a blow-up. In the proposition that follows, recall that $Q =132312$ is a shorthand for $Q = (s_1,s_3,s_2,s_3,s_1,s_2)$ (see Section 1.2). 

\begin{prop}
 The Bott-Samelson resolution $\varphi_Q$ for $Q=132312$ is not a blow-up map.
\end{prop}

\begin{proof}

 It can be checked that $BS^Q_{ni}$ is a union of three codimension 1 sub Bott-Samelson varieties, namely
 
 $$BS^{13-312}\text{, } BS^{132-12} \text{ and } BS^{1323-2}.$$
 
 Suppose that $\varphi_Q$ is a blow-up of $X^{4321}$ along some center $C$. Since any blow-up is an isomorphism away from its center, $C$ necessarily contains $\varphi_Q(BS^Q_{ni})$. Then $\varphi^{-1}_Q(C)$ is the union of the three $BS^R$ listed above and possibly some other divisor $E$. Now consider the following commuting diagram
 \[
 \begin{tikzcd}
 BS^{Q} \arrow{r}{\psi} \arrow[swap]{dr}{\varphi_{Q}} & BS^{Q'} \arrow{d}{\varphi_{Q'}} \\
 & X^{4321}
 \end{tikzcd}
 \]
 where $Q' =(s_1s_3s_2s_3,s_1,s_2)$. It's not hard to check that $\varphi_{Q'}^{-1}(C)$ is the union of two codimension 1 sub Bott-Samelson varieties, namely $BS^{R'}$ for $R'=(s_1s_3s_2,s_1,s_2),(s_1s_3s_2s_3,e,s_2)$, as well as $\psi(E)$ (the map $\psi$ has collapsed down the $BS^{13-312}$ component of $BS^Q_{ni}$). Note that both $BS^Q$ and $BS^{Q'}$ are smooth, so all the above divisors are automatically Cartier. By the universal property of blow-up maps, there exists a fibre diagram:
  
 \[
 \begin{tikzcd}
 BS_{ni}^{Q'}\cup \psi(E) \arrow{r} \arrow[swap]{d} & BS^{Q'} \arrow{d}{\pi} \\
 BS_{ni}^Q\cup E \arrow[swap]{d}\arrow{r} & BS^Q\arrow{d}{\varphi_Q}\\
 C \arrow{r} & X^w
 \end{tikzcd}
 \]
 
 Since $BS^Q$ and $BS^{Q'}$ have the same dimension, the preimage of a divisor under $\pi$ must again be a divisor (it clearly cannot be all of $BS^{Q'}$). Therefore, the preimage of each irreducible component of $BS^Q_{ni}$ must be a union of irreducible components in $BS^{Q'}_{ni}$. There are however 3 components in $BS^{Q}_{ni}$ and only 2 components in $BS^{Q'}_{ni}$, a contradiction.

\end{proof}

Note that  $BS^Q_{ni}\cup E$ from the proposition is an example of a choice for $D$ from the beginning of the section. Interestingly enough, choosing another reduced expression $Q$ for $4321$ will avoid many of the problems mentioned in Proposition 3.2.

Since we cannot draw a similar factorization diagram for $BS^{Q'}$ as we did for $BS^Q$, it might be true that $\varphi_{Q'}$ is a blow-up, and that choosing our Cartier divisor $D$ with more than one component is acceptable (although this is not immediately clear).

Let us now consider those cases for which an irreducible $D$ can be chosen. If $Y = \varphi_Q(D)$ is a Cartier divisor in $X^w$, the blow-up of $X^w$ along $Y$ is an isomorphism and clearly won't be the Bott-Samelson resolution (except in the rare instance that $\varphi_Q$ was already an isomorphism). The next example shows how obvious choices of $D$ can potentially lead to this problem.

\begin{example} \normalfont In Section \ref{Resolution Algorithm}, we saw that $\varphi_Q$ for $Q={(r_1,r_2r_3r_2r_1)}$ was a strict Bott-Samelson resolution for $X^{4231}$ (whose singular locus is $X^{2143}$). The natural choice for any $D$ from Proposition 3.1 is some combination of codimension 1 sub Bott-Samelson varieties. These $B$-invariant divisors after all generate the effective cone of $BS^{Q}$ since $Q$ is reduced (see \cite{Anderson} for commentary on divisor classes of Bott-Samelson varieties). It is not hard to check however that $R$ is reduced for any codimension 1 $BS^R$ that contains the preimage of $X^{2143}$. In other words, $\codim(BS^Q_{ni})>1$ and any choice of $D$ will result in $\codim(\varphi_Q(D))=1$ by Proposition 3.1. Since $X^{4231}$ is not factorial, we must consider whether $\varphi_Q(D)$ is a Cartier divisor. One can verify that the image of any codimension 1 $BS^R$ is in fact not Cartier. We have thus found an irreducible $D$ containing $BS^Q_{ni}$ for which $\varphi_Q(D)$ is a viable center. \end{example}

 On the other hand, if $Y$ is not Cartier then the blow-up map is not trivial. By the proposition, we have the following fibre diagram:

\[
\begin{tikzcd}
D \arrow{r} \arrow[swap]{d} & BS^Q \arrow{d}{\pi} \\
E \arrow[swap]{d}\arrow{r} & Bl_Y(X^w)\arrow{d}\\
Y \arrow{r} & X^w
\end{tikzcd}
\]

Then both $D$ and $E$ are irreducible Cartier divisors. Since $\varphi_Q$ and the blow-up map are proper and birational, $\pi$ must also be proper and birational. If we know that $Bl_Y(X^w)$ is normal (which is equivalent to  $\mathcal{I}(Y)^k$ being integrally closed for all $k>0$), then any finite birational map to it must be an isomorphism. Under these conditions, the problem has been reduced to checking whether $\pi$ is quasi-finite to conclude that $\varphi_Q$ is isomorphic to a blow-up. This boils down to understanding the fibres of $\varphi_Q$ and comparing that to the fibres of the blow-up. Even this question proves difficult for all but the simplest cases. See \cite{Gaussent} and \cite{Gaussent2} for more information on the fibres of a Bott-Samelson resolution.

Let us conclude this section with a final thought. If $D$ can be chosen to be $BS^{R}$ where $R$ is not reduced, then $Y=\varphi_Q(D)$ is a Schubert subvariety since $\varphi_Q$ is $B$-equivariant and $BS^R$ is $B$-invariant. The question of blowing-up $X^w$ along a $B$-invariant subspace might provide extra information by extending the $B$-action to the total transform. This prompts the question, when is a Hironaka resolution of $X^w$ actually $B$-equivariant? Such an equivariant resolution of singularities exists and is discussed in Proposition 3.9.1 of \cite{Kollar}.

\section{Strict Bott-Samelson Resolutions}\label{strict}

In this section we will drop our previous attempts to view $\varphi_Q$ as a blow-up, and instead focus on whether it is possible to resolve the singularities of $X^w$ by a sequence of strict $\varphi_Q$ maps. Let us first show that it is enough to restrict our attention to generalized $BS^Q$ maps for $Q=(w_1,w_2)$. 

Let $R=(v_1,v_2,v_3)$. We can view $\varphi_R$ as $\varphi_{R''}\circ(\varphi_{R'}\widetilde{\times} \id_{v_3})$ where $R' = (v_1,v_2)$ and $R'' =(v_1*v_2, v_3)$ (each is surjective). The map $\varphi_{R'}\widetilde{\times} \id_{v_3}$ takes $[g_1,g_2,g_3]$ to $[g_1g_2,g_3]$ (the notation $\widetilde{\times}$ is taken from the theory of affine Grassmannians which we will not stop to define here). Then $\varphi_R$ being strict implies that $\varphi_{R'}$ and $\varphi_{R''}$ are also strict. The same is true if $R' = (v_2,v_3)$ and $R''=(v_1,v_2*v_3)$. By induction we get the following result:

\begin{lem}
 Let $X^w$ be a singular Schubert variety. Any strict Bott-Samelson resolution of $X^w$ by a sequence of $\varphi_{Q_k}$ can be factored into a sequence of strict $\varphi_{R_i}$ maps where each $R_i= (v^{i}_1,v^i_2)$.
\end{lem}

In other words, it is enough to check strictness for $\varphi_Q$ where $Q$ is a 2-tuple. Given a Schubert variety $X^w$, and given a reduced $w = s_{i_1}s_{i_2}...s_{i_p}$, we ask whether we  can we write $w=w_1*w_2$ where $w_1 = s_{i_1}...s_{i_q}$ and $w_2=s_{i_{q+1}}...s_{i_p}$ such that

$$\overline{Bw_1B}\times^B \overline{Bw_2B}/B \rightarrow X^w$$
is strict? Each $\overline{Bw_iB}/B$ is a Schubert variety, and if it is not smooth we hope that we can continue this process until we have a smooth Bott-Samelson variety. Therefore, we wish to find

$$\varphi_{Q_i}: BS^{Q_i}\rightarrow \overline{Bw_iB}/B$$
for each $i$, where the product of these two maps provides a strict map onto $\overline{Bw_1B}\times^B \overline{Bw_2B}/B$. Notice however that the $Q_i$ might use a different reduced form than that used in $w_i$. This process is highly dependent on how we write the word in simple reflections as only certain reduced words will provide strict resolutions. For this reason finding a strict resolution of singularities is not just about finding some $\overline{Bw_1B}\times^B ...\times^B \overline{Bw_kB}/B$ with the usual map onto $X^w$, but rather some sequence of Bott-Samelson maps whose composition may not be the standard product map $\varphi_Q$.

Let us make one final reduction which will make it possible to use a computer search to find strict $\varphi_Q$. We will first need a technical lemma which we thank Allen Knutson for providing.

\begin{lem}
Suppose that $X\hookrightarrow \mathbb{P}^n$ is an equivariant inclusion with respect to a torus action $T$. If $\dim X>0$ (and $X\neq \emptyset$), then $|X^T|>1$.
\end{lem}
\begin{proof}
	If $X=X^T$, then the result is obvious. Otherwise, choose an $x\in X\setminus X^T$. Let $Y = \overline{T\cdot x}$ and $\nu:\tilde{Y}\rightarrow Y$ be the normalization of $Y$. Then $\tilde{Y}$ is a projective toric variety. Furthermore, since $x$ is not a fixed point by assumption, $\dim \tilde{Y}\geq 1$.
	
	We wish to use the $T$-fixed points of $\tilde{Y}$ to show the existence of at least two $T$-fixed points for $Y$. Let $\mathcal{L}$ be the pullback of the sheaf $\mathcal{O}(1)$ in the diagram:
	\[
	\begin{tikzcd}
	\mathcal{L} \arrow{r} \arrow[swap]{d} & \mathcal{O}(1) \arrow{d} \\
	\tilde{Y} \arrow{r}{\nu} & Y
	\end{tikzcd}
	\]
	Since $\nu$ is a finite map, $\mathcal{L}$ is an ample line bundle over $\tilde{Y}$. Then $\tilde{Y}$ is a toric variety associated to a lattice polytope $P$ (see \cite[Thm 6.2.1]{CLS}). In fact, $P$ is the convex hull of the weights of the action of $T$ on the fibres $\mathcal{L}|_f$ for $f\in \tilde{Y}^T$. 
	
	Note that the weight of the $T$-action on $\mathcal{L}|_f$ is the same as the weight of $T$ acting on $\mathcal{O}(1)|_{\nu(f)}$. Therefore, distinct $T$-fixed points of $\tilde{Y}$ correspond to distinct weights for the $T$-action on the line bundles, which in turn corresponds to distinct $T$-fixed points for $Y$. It is enough then to show that $|\tilde{Y}^T|>1$ to imply that $|Y^T|>1$.
	
	By the orbit-cone correspondence \cite[Thm 3.2.6]{CLS}, since $\dim\tilde{Y}>0$, $|\tilde{Y}^T|>1$ as required.
\end{proof}

\begin{lem}
	Let $\varphi_{Q}$ be a Bott-Samelson resolution onto $X^w$ with $Q$ reduced. Then $\varphi_Q$ is an isomorphism iff it is a bijection on $T$-fixed points.
\end{lem}
\begin{proof}
	One direction is obvious. For the other direction, observe that $\varphi_Q$ is birational and $X^w$ is normal, so it is enough to show that $\varphi_Q$ is finite to conclude that it is an isomorphism. Since $\varphi_Q$ is also proper, we only need that $\varphi_Q$ is quasi-finite for the result to hold.
	
	Consider the fibre $Z_v = \varphi_Q^{-1}(e_v)$, which is a projective subvariety of $BS^Q$. The embedding of $Z_v$ into projective space is equivariant with respect to the torus action on $Z_v$ (since $\varphi_Q$ is $B$-equivariant). Under these conditions, dim($Z_v)>0 \Rightarrow |Z_v^T|>1$ by Lemma 4.2. Since $|Z_v^T|=1$ by assumption, we conclude that dim($Z_v)=0$.
	
	Since $\varphi_Q$ is $B$-equivariant, the fibre of a point in $Be_vB/B$ is a $B$-translate of $Z_v$.  These Bruhat cells cover $X^w$, so $\varphi_Q$ is quasi-finite.
\end{proof}

It is therefore enough to check strictness on $T$-fixed points. The map

$$\overline{Bw_1B}\times^B \overline{Bw_2B}/B \rightarrow X^w$$
will be strict iff the product map on Bruhat intervals is strict:

$$[e,w_1]\times[e,w_2] \rightarrow [e,w].$$

\begin{example}
	\normalfont Let us try and strictly resolve $X^{3412}$, the other singular Schubert variety for the $n=4$ case (the variety $X^{4231}$ was discussed in section \ref{Resolution Algorithm}). The word $s_2s_1s_3s_2$ is a reduced expression for this permutation. Taking $Q=(s_2,s_1s_3s_2)$, we want to know if $\varphi_Q$ is strict. By Lemma 4.3, it is enough to check this on $T$-fixed points, so consider the product map

$$[e,s_2]\times[e,s_1s_3s_2] \rightarrow [e,s_2s_1s_3s_2].$$

The map is obviously 2-to-1 over $\{e,s_2\}$. Checking the cardinalities of the intervals, it is not hard to see that the map is a bijection everywhere else. By \cite{Woo-Yong}, $X^{3412}$ is only singular on $X^{1324}$, or $\{e,s_2\}$ in terms of reduced words for the $T$-fixed points. Therefore $X^{3412}$ can be resolved strictly.
\end{example}

The problem of finding strict Bott-Samelson resolutions has its own difficulties, and like the previous section, the answer is not very satisfying. We would like to say that there always exists a $Q=(w_1,w_2)$ such that $\varphi_Q$ is strict. This is not the case, however, since there are permutations for which any choice of $Q$ will not be strict (see below). This leads to a further complication. What if initially there exists a $Q=(w_1,w_2)$ such that $\varphi_Q$ is strict, but then $w_2$ is a reduced word for one of these non-strict resolutions?

Another fundamental barrier to proving general results in this section is the lack of tools for identifying the singularities of $X^w$ based on a reduced word for the permutation $w$. Indeed interval pattern embeddings for permutations don't translate very well into criteria involving reduced words. Therefore, the verification of which permutations possess a strict desingularization had to proceed using a computer search.

\subsection{The $n=5$ Case}\label{n5}

Consider $X^w$ for $w\in S_5$. There are certainly more cases to consider than the two from $n=4$. A computer check through all the possibilities (this potentially means every division of every reduced word for $w$) reveals that the only problematic Schubert variety is $X^{45312}$. The permutation $45312$ can be written in simple reflections as $s_2s_3s_2s_1s_4s_2s_3s_2$, and no division will provide a strict resolution. We can check that $X^{45312}$ is smooth at $X^{21345}$ and $X^{12354}$ (in reduced words this is $s_1$ and $s_4$ respectively). 

Any division of  $s_2s_3s_2s_1s_4s_2s_3s_2$ into two will produce a $\varphi_Q$ with a fibre over $s_4$. This is simply because any choice of $Q=(w_1,w_2)$ must satisfy  $s_2 \leq w_1,w_2$. Therefore, $(s_2,s_4s_2)$ or $(s_2s_4,s_2)$ will be a fibre over $s_4$ (in addition to the obvious $(e,s_2)$ or $(s_2,e)$). A computer check shows that no other reduced word works. In fact, it is not too hard to convince yourself of this by hand.

A priori, it is possible for some other permutation $v$ to be divided in such a way that the only resolution is $$\overline{Bv_1B}\times^B \overline{Bv_2B}/B \rightarrow X^v$$ where $\overline{Bv_1B}/B = X^{45312}$ (and therefore no subsequent divisions are possible). However, this would mean that some subword of $v$ was a reduced word for $45312$. That is, it would require $v > v_1$ in the Bruhat order. 

Luckily, the only permutations above $w$ in Bruhat order are $54312,\\
 45321, 54321$ which are all smooth already (and hence wouldn't be subject to this division process). We have thus shown:

\begin{prop}
Let $X^w$ be a singular Schubert variety for $w\in S_5$. There exists a strict Bott-Samelson resolution of singularities for $X^w$ iff  $w\neq 45312$.
\end{prop}

\subsection{The $n=6$ Case}\label{n6}

This case is substantially harder than the last, and even a computer check can take a long time to verify all possibilities. The list of Schubert varieties $X^w$ for which no division of any reduced word for $w$ works in defining a strict $\varphi_Q$ are for 

$$w = 156423, 453126, 456312, 465132, 465312, 546213, 546312,$$
 $$ 564123, 564132, 564213, 564312, 632541, 653421.$$

Unlike the $n=5$ case, there are indeed  $v \geq w$ for which $X^v$ is not smooth. Hence if we started subdividing a word for $v$, we may eventually encounter a $w$ on the list that could not be strictly resolved. To avoid this complication, we will only consider permutations which are not above any of the $w$ on the list.

Because of this simplification, we can now reduce our list to 
$$u= 156423, 453126, 632541$$
since any other $w$ in the list above satisfies $w\geq u$ for some $u$. We therefore have a weaker result than that for $n=5$:

\begin{prop}
Let $X^w$ be a singular Schubert variety for $w\in S_6$. There exists a strict Bott-Samelson resolution of singularities for $X^w$ if  $w\ngeq 156423, 453126 \text{ or } 632541$.
\end{prop}

A natural question to ask is whether the condition in the proposition is a closed one. In other words, is it true that $w\geq 156423, 453126 \\ \text{ or } 632541$ results in $X^w$ not having a strict Bott-Samelson resolution? Unfortunately this is not the case.

\begin{example}
	\normalfont  Consider the singular Schubert variety $X^{456123}$. Even though $156423 \leq 456123$, $X^{456123}$ still admits a strict Bott-Samelson resolution of singularities starting with $\varphi_Q$ for $Q = (s_3s_4s_5s_2s_3,s_4s_1s_2s_3)$.

While the second factor $\overline{Bs_4s_1s_2s_3B}/B$ of $BS^Q$ is smooth, the first is not, but can be strictly resolved using $\varphi_R$ with $R=(s_3,s_4s_5s_2s_3)$. 

There are of course $w > 156423, 453126 \text{ or } 632541$ for which $X^w$ cannot be strictly resolved. For example, the only strict Bott-Samelson map  onto $X^{256413}$ using $\varphi_Q$ is with $Q = (s_1, w')$ where $w'$ is some reduced expression for $X^{156423}$ (which we know cannot be strictly resolved). 
\end{example}

As the feasibility of continuing this process diminishes with each increase of $n$, it might be prudent to step back and try to find a pattern.

One observation is that $156423$ and $453126$ both have embedded versions of $45312$. Indeed, given a permutation $\pi$ in $S_m$, we can define a new permutation $\tau(k) \in S_n$ (for $n>m$) by pattern embedding $\pi$ into $123\cdots n$ starting at position $k$ and ending in position $k+m$. We will call this a translation of $\pi$ in $S_n$. For example, if $\pi =45312$, then $\tau(1)=453126$ and $\tau(2) = 156423$. It is not a surprise that translations of problematic permutations are also a problem for higher $n$ as their reduced expressions in $S_n$ still possess the same division problems as they did in $S_m$.

In actuality, the only new permutation for $n=6$ is $632541$. This permutation could be formed by inverting the first and last value, the second and third value, and the third-last and second-last values (of course this is just one interpretation). In $S_n$ this can be written as $\pi_n=n324\cdots(n-3)(n-1)(n-2)1$ for $n>6$ with $\pi_6 = 632541$. Let us now define a list $\Gamma_n$ of permutations we want to avoid. We know $\Gamma_5 = \{45312\}$. Then in general let

$$\Gamma_n = \{\pi_n\}\cup \Gamma^T_{n-1}$$
where $\Gamma^T_{n-1}$ are the translations of permutations in $\Gamma_{n-1}$. For example, 

$$\Gamma_7 = \{1267534, 1564237, 1743652, 4531267, 6325417, 7324651\}$$

 It is worth noting that $|\Gamma_{n+4}| = \frac{n(n+1)}{2}$. Indeed, $\Gamma_{n+4}$ is just the union of the translations of $45312$ and $\pi_6$ up to $\pi_{n+4}$. There are $n$ translations of $45312$ and $n+5-k$ translations of $\pi_k$. This gives a new interpretation of the triangular numbers. Many thanks to the reviewer of this article for making this observation.

We can now formulate a conjecture to generalize our work so far

\begin{conj}
The singular Schubert variety $X^w$ for $w\in S_n$ with $n>4$ can be desingularized using a sequence of strict Bott-Samelson resolutions if $w \ngeq v$ for $v\in \Gamma_n$.
\end{conj}

\begin{thm}
The conjecture is true for $5\leq n\leq 8$.
\end{thm}

Future work towards this problem includes developing techniques for verifying the conjecture without the need for a computer program. This would mean a deeper understanding of how singularities can be more easily read from a reduced word for the permutation. 

The precise geometric description of the singular locus given in \cite{Cortez} might prove useful in studying strict resolutions. In fact, strict resolutions for permutations in $\Gamma_n$ might exist using the constructions in \cite{Cortez2} and \cite{Cortez}. It is also interesting to note that some of the permutations in $\Gamma_n$ appear in the interval pattern embedding description of the singular locus in \cite{Woo-Yong}.  

Furthermore, a better understanding of how Bott-Samelson resolutions compare with a Hironaka resolution may provide a clearer path to creating hybrid desingularization algorithms for similar combinatorial objects.

\bibliography{FinalStrictSchubert_ArxivVersion}

\end{document}